\RequirePackage{fix-cm}
\documentclass{svjour3}                     
\usepackage{graphicx}
\usepackage{mathptmx}

\usepackage{latexsym, epsfig,cite, psfrag,eepic,colordvi,bm}
\usepackage{graphics,color}
\usepackage{amsmath,amsfonts,amssymb,amscd}
\usepackage[all]{xy}

\begin{document}
\title{A simple framework on sorting permutations 
}


\author{Ricky X. F. Chen         \and
        Christian M. Reidys 
}


\institute{Ricky X. F. Chen \and Christian M. Reidys$^*$\at
              Department of Mathematics and Computer Science, University of Southern Denmark,
              Campusvej 55, DK-5230 Odense M, Denmark\\
              Tel.$^*$: +45-24409251\\
              Fax$^*$: +45-65502325\\
              \email{chen.ricky1982@gmail.com (Ricky X. F. Chen)} \\
              \email{duck@santafe.edu (Christian M. Reidys)}        %
}

\date{Received: date / Accepted: date}

\maketitle

\begin{abstract}
In this paper we present a simple framework to study various distance problems
of permutations, including the transposition and block-interchange distance of permutations
as well as the reversal distance of signed permutations.
These problems are very important in the study of the evolution of genomes.
We give a general formulation for lower bounds of the transposition and block-interchange
distance from which the existing lower bounds obtained by Bafna and Pevzner, and Christie
can be easily derived. As to the reversal distance of signed permutations,
we translate it into a block-interchange distance problem of permutations so that
we obtain a new lower bound. Furthermore, studying distance problems via our framework
motivates several interesting combinatorial problems related to product of permutations,
some of which are studied in this paper as well.

\end{abstract}

\section{Introduction}
Let $\mathcal{S}_n$ denote the group of permutations, i.e.~the group
of bijections from $[n]=\{1,\dots,n\}$ to $[n]$, where the multiplication
is the composition of maps.
We shall discuss the following three representations of a permutation $\pi$ on $[n]$:\\
\emph{two-line form:} the top line lists all elements in $[n]$, following the natural order.
The bottom line lists the corresponding images of the elements on the top line, i.e.
\begin{eqnarray*}
\pi=\left(\begin{array}{ccccccc}
1&2& 3&\cdots &n-2&{n-1}&n\\
\pi(1)&\pi(2)&\pi(3)&\cdots &\pi({n-2}) &\pi({n-1})&\pi(n)
\end{array}\right).
\end{eqnarray*}
\emph{one-line form:} $\pi$ is represented as a sequence $\pi=\pi(1)\pi(2)\cdots \pi(n-1)\pi(n)$.\\
\emph{cycle form:} regarding $\langle \pi\rangle$ as a cyclic group, we represent $\pi$ by its
collection of orbits (cycles).
The set consisting of the lengths of these disjoint cycles is called the cycle-type of $\pi$.
We can encode this set into a non-increasing integer sequence $\lambda=\lambda_1
\lambda_2\cdots$, where $\sum_i \lambda_i=n$, or as $1^{a_1}2^{a_2}\cdots n^{a_n}$, where
we have $a_i$ cycles of length $i$.
A cycle of length $k$ will be called a $k$-cycle. A cycle of odd and even length will be called an odd
and even cycle, respectively. It is well known that all permutations of a same cycle-type
form a conjugacy class of $\mathcal{S}_n$.

In~\cite{ChR1}, plane permutations were used to study one-face hypermaps.
In particular, a combinatorial proof for the following result of Zagier~\cite{zag} and
Stanley~\cite{stan3} was presented:
the number of $n$-cycles $\omega$, for which
$\omega(12\cdots n)$ has exactly $k$ cycles is $0$,
if $n-k$ is odd and $\frac{2C(n+1,k)}{n(n+1)}$, otherwise, where $C(n,k)$
is the unsigned Stirling number of the first kind.
In this paper, we employ the framework of plane permutations to study sorting permutations,
which is motivated by the transposition action on plane permutations defined in~\cite{ChR1}.
This ties to important problems in the context of bioinformatics,
in particular the evolution of genomes by rearrangements in DNA as well as RNA.
For the related studies and general biological background, we refer to
~\cite{pev1,pev2,pev3,pev4,chri1,chri2,bul,braga,eli,yan} and references therein.

An outline of this paper is as follows. In section $2$, we present a short introduction
on plane permutations, especially we study a natural action on plane permutations which is
related to various sorting operations of permutations. In section $3$,
we study the transposition distance problem of permutations and derive a generally formulated lower bound
which implies the lower bound obtained by Bafna and Pevzner~\cite{pev1} via the cycle-graph model.
Our formula of the lower bound motivates several
optimization problems as well. One of them is to determine $\max_{\gamma} |C(\alpha\gamma)-C(\gamma)|$
for a fixed permutation $\alpha$, where $C(\pi)$ denotes the number of cycles in the permutation $\pi$.
We will solve this optimization problem in this paper.
In Section $4$, we consider the block-interchange distance of permutations and establish the block-interchange
distance formula due to Christie~\cite{chri2}.
In Section $5$, we study the reversal distance of signed permutations.
By translating the reversal distance of signed
permutations into block-interchange distance of permutations with restricted block-interchanges, we
prove a new formula on the lower bound of the reversal distance.
We then observe that this bound is typically equal to the reversal distance.


\section{Plane permutations}

\begin{definition}[Plane permutation]
A plane permutation on $[n]$ is a pair $\mathfrak{p}=(s,\pi)$ where $s=(s_i)_{i=0}^{n-1}$
is an $n$-cycle and $\pi$ is an arbitrary permutation on $[n]$. The permutation $D_{\mathfrak{p}}=
s\circ \pi^{-1}$ is called the diagonal of $\mathfrak{p}$.
\end{definition}\label{2def1}

Given $s=(s_0s_1\cdots s_{n-1})$,
a plane permutation $\mathfrak{p}=(s,\pi)$ can be represented by two aligned rows:
\begin{equation}
(s,\pi)=\left(\begin{array}{ccccc}
s_0&s_1&\cdots &s_{n-2}&s_{n-1}\\
\pi(s_0)&\pi(s_1)&\cdots &\pi(s_{n-2}) &\pi(s_{n-1})
\end{array}\right)
\end{equation}
Indeed, $D_{\mathfrak{p}}$ is determined by the diagonal-pairs (cyclically) in the two-line
representation here, i.e., $D_{\mathfrak{p}}(\pi(s_{i-1}))=s_i$ for $0<i< n$, and
$D_{\mathfrak{p}}(\pi(s_{n-1}))=s_0$.

Given a plane permutation $(s,\pi)$ on $[n]$ and a sequence $h=(i,j,k,l)$, such that $i\leq j<k \leq l$
and $\{i,j,k,l\}\subset [n-1]$, let
$$
s^h=(s_0,s_1,\dots,s_{i-1},\underline{s_k,\dots,s_l},s_{j+1},\dots,s_{k-1},
\underline{s_i,\dots,s_j},s_{l+1},\dots),
$$
i.e.~the $n$-cycle obtained by transposing the blocks $[s_i,s_j]$ and $[s_k,s_l]$ in $s$.
Note that in case of $j+1=k$, we have
$$
s^h=(s_0,s_1,\dots,s_{i-1},\underline{s_k,\dots,s_l},\underline{s_i,\dots,s_j},s_{l+1},\dots).
$$
Let furthermore
$$
\pi^h=D_{\mathfrak{p}}\circ s^h,
$$
that is, the derived plane permutation, $(s^h,\pi^h)$, can be represented as

\begin{eqnarray*}
\left(
\vcenter{\xymatrix@C=0pc@R=1pc{
\cdots s_{i-1}\ar@{->>}[d]  & s_k\ar@{--}[dl] &\cdots & s_l\ar@{--}[dl]\ar@{->>}[d] & s_{j+1} &\hspace{-0.5ex}\cdots\hspace{-0.5ex} &
s_{k-1}\ar@{->>}[d] & s_i\ar@{--}[dl] &\cdots & s_{j}\ar@{--}[dl]\ar@{->>}[d] & s_{l+1}  \cdots\\
\cdots \pi(s_{k-1}) & \pi(s_k) & \cdots & \pi(s_j) & \pi(s_{j+1}) &\hspace{-0.5ex}\cdots\hspace{-0.5ex} & \pi(s_{i-1}) & \pi(s_i) &\cdots & \pi(s_l)& \pi(s_{l+1}) \cdots
}}
\right).
\end{eqnarray*}
We write $(s^h,\pi^h)= \chi_h\circ(s,\pi)$.
Note that the bottom row of the two-row representation of $(s^h,\pi^h)$
is obtained by transposing the blocks $[\pi(s_{i-1}),\pi(s_{j-1})]$ and
$[\pi(s_{k-1}),\pi(s_{l-1})]$ of the bottom row of $(s,\pi)$.
In the following, we refer to general $\chi_h$ as block-interchange and for the special case of $k=j+1$,
we refer to $\chi_h$ as transpose.
As a result, we observe
\begin{lemma}\label{2lemx1}
Let $(s,\pi)$ be a plane permutation on $[n]$ and $(s^h,\pi^h)=\chi_h \circ (s,\pi)$ for $h=(i,j,k,l)$.
Then, $\pi(s_r)=\pi^h(s_r)$ if $r\in \{0,1,\ldots,n-1\}\setminus \{i-1,j,k-1,l\}$, and
for $j+1<k$
\begin{equation*}
\pi^h(s_{i-1})=\pi(s_{k-1}), \quad \pi^h(s_j)=\pi(s_l), \quad \pi^h(s_{k-1})=\pi(s_{i-1}), \quad
\pi^h(s_l)=\pi(s_j)
\end{equation*}
hold and for $j=k-1$ we have
$$
\pi^h(s_{i-1})=\pi(s_{j}), \quad \pi^h(s_j)=\pi(s_l), \quad \pi^h(s_l)=\pi(s_{i-1}).
$$
\end{lemma}

We shall proceed by analyzing the induced changes of the $\pi$-cycles when passing
to $\pi^h$. By Lemma~\ref{2lemx1}, only the $\pi$-cycles containing $s_{i-1}$, $s_{j}$, $s_{k-1}$,
$s_l$ will be affected.

\begin{lemma}\label{2lem4}
Let $(s^h, \pi^h)=\chi_h \circ (s, \pi)$, where $h=(i,j,j+1,l)$.
Then there exist the following six scenarios for the pairs $(\pi,\pi^h)$:

\begin{center}
\begin{tabular}{|c|c|c|}
\hline
Case~$1$ &$\pi$& $(s_{i-1},v_1^i,\ldots v_{m_i}^i),  (s_j, v_1^j,\ldots v_{m_j}^j), (s_l,v_1^l,\ldots v_{m_l}^l)$\\\cline{2-3}
&$\pi^h$&$ (s_{i-1},v_1^j,\ldots v_{m_j}^j, s_{j}, v_1^l,\ldots v_{m_l}^l, s_l, v_1^i, \ldots v_{m_i}^i)$\\
\hline\hline
Case~$2$ &$\pi$& $(s_{i-1},v_1^i,\ldots v_{m_i}^i, s_l, v_1^l,\ldots v_{m_l}^l, s_j, v_1^j, \ldots v_{m_j}^j)$\\\cline{2-3}
&$\pi^h$& $(s_{i-1}, v_1^j,\ldots v_{m_j}^j),  (s_j, v_1^l,\ldots v_{m_l}^l), (s_l, v_1^i, \ldots v_{m_i}^i)$\\
\hline\hline
Case~$3$ &$\pi$& $(s_{i-1},v_1^i,\ldots v_{m_i}^i, s_j, v_1^j, \ldots v_{m_j}^j, s_l, v_1^l,\ldots v_{m_l}^l)$\\\cline{2-3}
&$\pi^h$& $(s_{i-1},v_1^j,\ldots v_{m_j}^j, s_l, v_1^i,\ldots v_{m_i}^i, s_j, v_1^l, \ldots v_{m_l}^l)$\\
\hline\hline
Case~$4$ &$\pi$& $(s_{i-1}, v_1^i, \ldots v_{m_i}^i, s_{j}, v_1^j, \ldots v_{m_j}^j), (s_l, v_1^l, \ldots v_{m_l}^l)$\\\cline{2-3}
&$\pi^h$&$ (s_{i-1}, v_1^j, \ldots v_{m_j}^j), (s_{j}, v_1^l, \ldots v_{m_l}^l, s_l, v_1^i, \ldots v_{m_i}^i)$\\
\hline\hline
Case~$5$ &$\pi$& $(s_{i-1}, v_1^i, \ldots v_{m_i}^i),  (s_{j}, v_1^j, \ldots v_{m_j}^j, s_l, v_1^l, \ldots v_{m_l}^l)$\\\cline{2-3}
&$\pi^h$&$ (s_{i-1}, v_1^{j},\ldots v_{m_j}^j, s_l, v_1^i, \ldots v_{m_i}^i), (s_{j}, v_1^l, \ldots v_{m_l}^l)$\\
\hline\hline
Case~$6$ &$\pi$& $(s_{i-1}, v_1^i, \ldots v_{m_i}^i, s_l, v_1^l, \ldots v_{m_l}^l), (s_{j}, v_1^j, \ldots v_{m_j}^j)$\\\cline{2-3}
&$\pi^h$&$ (s_{i-1}, v_1^j, \ldots v_{m_j}^j, s_{j}, v_1^l, \ldots v_{m_l}^l), (s_l, v_1^i, \ldots v_{m_i}^i)$\\
\hline
\end{tabular}
\end{center}
\end{lemma}
\begin{proof} We shall only prove Case~$1$ and Case~$2$, the remaining four cases
can be shown analogously.
For Case~$1$, the $\pi$-cycles containing $s_{i-1}$,~$s_j$,~$s_l$ are
$$
(s_{i-1},v_1^i,\ldots v_{m_i}^i),  (s_j, v_1^j,\ldots v_{m_j}^j), (s_l,v_1^l,\ldots v_{m_l}^l).
$$
Lemma~\ref{2lemx1} allows us to identify the new cycle structure by inspecting the critical
points $s_{i-1}$, $s_j$ and $s_l$.
Here we observe that all three cycles merge and form a single $\pi^h$-cycle
\begin{eqnarray*}
(s_{i-1}, \pi^h(s_{i-1}),(\pi^h)^2(s_{i-1}),\ldots)&=&(s_{i-1}, \pi(s_j), \pi^2(s_j),\ldots )\\
&=& (s_{i-1},v_1^j,\ldots v_{m_j}^j, s_{j}, v_1^l,\ldots v_{m_l}^l, s_l, v_1^i, \ldots v_{m_i}^i).
\end{eqnarray*}
For Case $2$, the $\pi$-cycle containing $s_{i-1}$,~$s_j$,~$s_l$ is
$$
(s_{i-1},v_1^i,\ldots v_{m_i}^i, s_l, v_1^l,\ldots v_{m_l}^l, s_j, v_1^j, \ldots v_{m_j}^j).
$$
We compute the $\pi^h$-cycles containing $s_{i-1}$, $s_j$ and $s_l$ in $\pi^h$ as
\begin{eqnarray*}
(s_{i-1},\pi^h(s_{i-1}),(\pi^h)^2(s_{i-1}),\ldots)&=&(s_{i-1},\pi(s_{j}),
\pi^2(s_{j}),\ldots)=(s_{i-1}, v_1^j,\ldots v_{m_j}^j)\\
(s_j,\pi^h(s_j),(\pi^h)^2(s_j),\ldots)&=&(s_j,\pi(s_{l}),\pi^2(s_{l}),\ldots)=(s_j, v_1^l,\ldots v_{m_l}^l)\\
(s_l,\pi^h(s_l),(\pi^h)^2(s_l),\ldots)&=&(s_l,\pi(s_{i-1}),\pi^2(s_{i-1}),\ldots)=(s_l, v_1^i,\ldots v_{m_i}^i)
\end{eqnarray*}
whence the lemma.
\qed
\end{proof}

If we wish to express which cycles are impacted by a transpose of scenario $k$ acting on a plane
permutation, we shall say ``the cycles are acted upon by a Case~$k$ transpose''.

We next observe
\begin{lemma}\label{2lem3}
Let $\mathfrak{p}^h=\chi_h \circ \mathfrak{p}$ where $\chi_h$ is a transpose.
Then the difference of the number of cycles of $\mathfrak{p}$ and $\mathfrak{p}^h$ is even.
Furthermore the difference of the number of cycles, odd cycles, even cycles between
$\mathfrak{p}$ and $\mathfrak{p}^h$ is contained in $\{-2,0,2\}$.
\end{lemma}
\begin{proof}
Lemma~\ref{2lem4} implies that the difference of the numbers of cycles of $\pi$ and $\pi^h$
is even.
As for the statement about odd cycles, since the parity of the total number of elements
contained in the cycles containing $s_{i-1}$, $s_{j}$ and $s_l$ is preserved, the difference of the
number of odd cycles is even. Consequently, the difference of the number of even cycles is also
even whence the lemma.
\qed
\end{proof}

Suppose we are given $h=(i,j,k,l)$, where $j+1<k$. Then using the strategy of the proof
of Lemma~\ref{2lem4}, we have
\begin{lemma}\label{2lem5}
Let $(s^h, \pi^h)=\chi_h \circ (s, \pi)$, where $h=(i,j,k,l)$ and $j+1<k$.
Then, the difference of the numbers of $\pi$-cycles and $\pi^h$-cycles is contained in
$\{-2,0,2\}$. Furthermore, the scenarios, where the number of $\pi^h$-cycles increases by $2$, are
given by:
\begin{center}
\begin{tabular}{|c|c|c|}
\hline
Case~$a$ &$\pi$& $(s_{i-1}, v_1^i,\ldots v_{m_i}^i, s_j, v_1^j,\ldots v_{m_j}^j, s_{l}, v_1^l, \ldots v_{m_l}^l,s_{k-1}, v_1^k, \ldots v_{m_k}^k)$\\\cline{2-3}
&$\pi^h$&$(s_{i-1}, v_1^k, \ldots v_{m_k}^k),(s_j, v_1^l, \ldots v_{m_l}^l, s_{k-1}, v_1^i,\ldots v_{m_i}^i),(s_l, v_1^j, \ldots v_{m_j}^j) $\\
\hline\hline
Case~$b$ &$\pi$& $(s_{i-1}, v_1^i,\ldots v_{m_i}^i, s_{k-1}, v_1^k,\ldots v_{m_k}^k, s_{j}, v_1^j, \ldots v_{m_j}^j,s_l, v_1^l, \ldots v_{m_l}^l)$\\\cline{2-3}
&$\pi^h$&$(s_{i-1}, v_1^k, \ldots v_{m_k}^k, s_j, v_1^l, \ldots v_{m_l}^l), (s_{k-1}, v_1^i,\ldots v_{m_i}^i),(s_l, v_1^j, \ldots v_{m_j}^j) $\\
\hline\hline
Case~$c$ &$\pi$& $(s_{i-1}, v_1^i,\ldots v_{m_i}^i, s_{k-1}, v_1^k,\ldots v_{m_k}^k, s_{l}, v_1^l, \ldots v_{m_l}^l,s_j, v_1^j, \ldots v_{m_j}^j)$\\\cline{2-3}
&$\pi^h$&$(s_{i-1}, v_1^k, \ldots v_{m_k}^k, s_l, v_1^j, \ldots v_{m_j}^j), (s_{k-1}, v_1^i,\ldots v_{m_i}^i),(s_j, v_1^l, \ldots v_{m_l}^l) $\\
\hline\hline
Case~$d$ &$\pi$& $(s_{i-1}, v_1^i,\ldots v_{m_i}^i, s_l, v_1^l,\ldots v_{m_l}^l, s_{j}, v_1^j, \ldots v_{m_j}^j,s_{k-1}, v_1^k, \ldots v_{m_k}^k)$\\\cline{2-3}
&$\pi^h$&$(s_{i-1}, v_1^k, \ldots v_{m_k}^k),(s_j, v_1^l, \ldots v_{m_l}^l), (s_{k-1}, v_1^i,\ldots v_{m_i}^i, s_l, v_1^j, \ldots v_{m_j}^j) $\\
\hline\hline
Case~$e$ &$\pi$& $(s_{i-1}, v_1^i,\ldots v_{m_i}^i, s_{k-1}, v_1^k,\ldots v_{m_k}^k), (s_{j}, v_1^j, \ldots v_{m_j}^j,s_l, v_1^l, \ldots v_{m_l}^l)$\\\cline{2-3}
&$\pi^h$&$(s_{i-1}, v_1^k, \ldots v_{m_k}^k),(s_j, v_1^l, \ldots v_{m_l}^l), (s_{k-1}, v_1^i,\ldots v_{m_i}^i),(s_l, v_1^j, \ldots v_{m_j}^j) $\\
\hline
\end{tabular}
\end{center}

\end{lemma}

Now we are ready to study sorting permutations.
The main idea is to utilize various
block-transposition actions on plane permutations, motivated by the study of transposition
actions on the boundary component of fatgraphs \cite{reidys2}, where a topological
framework for studying reversal distance of signed permutations was presented.

\section{Transposition distance}

In this section, we shall use the one-line representation of permutations, i.e., we consider them to be
sequences. Given a sequence on $[n]$
$$
s=a_1\cdots a_{i-1}a_i\cdots a_ja_{j+1}\cdots a_k a_{k+1}\cdots a_{n},
$$
a transposition action on $s$ means to change $s$ into
$$
s'=a_1\cdots a_{i-1}a_{j+1}\cdots a_k a_i\cdots a_j a_{k+1}\cdots a_{n},
$$
for some $1\leq i\leq j<k\leq n$. Let $e_n=123\cdots n$.
The transposition distance of a sequence $s$ on $[n]$ is the minimum number of transpositions
needed to sort $s$ into $e_n$. Denote this distance as $td(s)$.

Let $C(\pi)$,~$C_{odd}(\pi)$ and $C_{ev}(\pi)$ denote the number of cycles, the number of odd cycles
and the number of even cycles in $\pi$, respectively. Furthermore, let $[n]^*=\{0,1,\ldots, n\}$, and
$$
\hat{e}_n=(0123\cdots n), \quad \bar{s}=(0,a_1,a_2,\cdots, a_n), \quad p_t=(n,n-1,\ldots,1,0).
$$

\begin{theorem}\label{5thm1}
\begin{eqnarray}
td(s)\geq \max_{\gamma}\left\{\frac{\max\{|C(p_t\bar{s}\gamma)-C(\gamma)|,|C_{odd}(p_t\bar{s}\gamma)-C_{odd}(\gamma)|,
|C_{ev}(p_t\bar{s}\gamma)-C_{ev}(\gamma)|\}}{2}\right\},
\end{eqnarray}
where $\gamma$ ranges over all permutations on $[n]^*$.
\end{theorem}

\begin{proof} 
For an arbitrary permutation $\gamma$ on $[n]^*$, $\mathfrak{p}=(\bar{s},\gamma)$ is a plane permutation.
By construction, each transposition on the sequence $s$ induces a transpose on $\mathfrak{p}$.
If $s$ changes to $e_n$ by a series of transpositions, we have, for some $\beta$, that $\mathfrak{p}$ changes into the
plane permutation $(\hat{e}_n, \beta)$.
By construction, we have
$$
D_{\mathfrak{p}}=\bar{s}\gamma^{-1}=\hat{e}_n\beta^{-1},
$$
and accordingly
$$
\beta=\gamma {\bar{s}}^{-1}\hat{e}_n.
$$
Since each transpose changes the number of cycles
by at most $2$ according to Lemma~\ref{2lem3},
at least $\frac{|C(\gamma {\bar{s}}^{-1}\hat{e}_n)-C(\gamma)|}{2}=
\frac{|C(p_t\bar{s}\gamma^{-1})-C(\gamma^{-1})|}{2}$
transposes are needed from $\gamma$ to $\beta$.
The same argument also applies to deriving the lower bounds in terms of odd and even cycles, respectively.
Note that $\gamma$ can be arbitrarily selected, then the proof follows.  \qed
\end{proof}

The most common model used to study transposition distance is cycle-graph proposed by Bafna and Pevzner~\cite{pev1}.
Given a permutation $s=s_1s_2\cdots s_n$ on $[n]$, the cycle graph $G(s)$ of $s$ is obtained
as follows:
add two additional elements $s_0=0$ and $s_{n+1}=n+1$. The vertices of $G(s)$ are the elements in $[n+1]^*$.
Draw a directed black edge from $i+1$ to $i$, and draw a directed gray edge from
$s_i$ to $s_{i+1}$, we then obtain $G(s)$.
An alternating cycle in $G(s)$ is a directed cycle, where its edges alternate in color.
An alternating cycle is called odd if the number of black edges is odd.
Bafna and Pevzner obtained lower and upper bound for $td(s)$ in terms of
the number of cycles and odd cycles of $G(s)$~\cite{pev1}.

By examining the cycle graph model $G(s)$ of a permutation $s$, it turns out
the cycle graph $G(s)$ is actually the directed graph representation of the
product $p_t\bar{s}$, if we identify the two auxiliary points $0$ and $n+1$.
The directed graph representation of a permutation $\pi$ is the directed graph
by drawing an directed edge from $i$ to $\pi(i)$. If we color the directed edge
of $\bar{s}$ gray and the directed edge of $p_t$ black, an alternating cycle
then determines a cycle of the permutation $p_t\bar{s}$. Therefore, the number of
cycles and odd cycles in $p_t\bar{s}$ is equal to the number of
cycles and odd cycles in $G(s)$, respectively.
As results, Theorem~\ref{5thm1} immediately implies

\begin{corollary}[Bafna and Pevzner~\cite{pev1}]\label{cor-bp}
\begin{eqnarray}
td(s)& \geq & \frac{n+1-C(p_t\bar{s})}{2},\\
td(s)& \geq & \frac{n+1-C_{odd}(p_t\bar{s})}{2}.
\end{eqnarray}
\end{corollary}

\begin{proof}
Setting $\gamma=(p_t\bar{s})^{-1}$ in Theorem~\ref{5thm1} implies the corollary. \qed
\end{proof}

In view of Theorem~\ref{5thm1}, employing an appropriate $\gamma$, it is possible to obtain a better lower
bound than what are in Corollary~\ref{cor-bp}.
This motivates the following problems: given a permutation $\pi$, what is the maximum number of
$|C(\pi\gamma)-C(\gamma)|$ (resp. $|C_{odd}(\pi\gamma)-C_{odd}(\gamma)|$, $|C_{ev}(\pi\gamma)-C_{ev}(\gamma)|$),
where $\gamma$ ranges over a set of permutations.

More generally, we can study the distribution functions
\begin{equation}
\sum_{\gamma\in A} z^{C(\pi\gamma)-C(\gamma)},\quad \sum_{\gamma\in A} z^{C_{odd}(\pi\gamma)-C_{odd}(\gamma)},
\quad \sum_{\gamma\in A} z^{C_{ev}(\pi\gamma)-C_{ev}(\gamma)},
\end{equation}
where $A$ is a set of permutations, e.g., a conjugacy class or all permutations.
In this paper, we will later determine $\max_{\gamma}\{|C(\pi\gamma)-C(\gamma)|\}$ for
an arbitrary permutation $\pi$. Surprisingly, the maximum for this case is achieved when
$\gamma=\pi^{-1}$ or $\gamma$ is the identity permutation.
For the other two problems in terms of odd cycle and
even cycle, we are unable to solve it at present.
Is it likely that the maxima are achieved when $\gamma=\pi^{-1}$ or $\gamma$ is the identity permutation
as well?

Another approach to obtain a better lower bound is fixing $\gamma$ and figuring out these unavoidable
transposes which do not increase (or decrease) the number of cycles (or odd, or even cycles) from
$\gamma$ to $\gamma {\bar{s}}^{-1}\hat{e}_n$. In particular, by setting $\gamma=p_t \bar{s}$,
it is not hard to analyze the number of ``hurdles" similar as in Christie~\cite{chri1} in the framework
of plane permutations,
which we do not go into detail here.

\section{Block-interchange distance of permutations}

A more general transposition problem, where the involved two blocks are not necessarily adjacent,
was studied in Christie~\cite{chri2}. It is referred to as the block-interchange distance problem.
The minimum number of block-interchanges needed to sort $s$ into $e_n$ is accordingly called the
block-interchange distance of $s$ and denoted as $bid(s)$.
Clearly, Lemma~\ref{2lem5} may facilitate the study of the block-interchange distance.
\begin{lemma}\label{5lem1}
Let $\mathfrak{p}=(\bar{s},\pi)$ be a plane permutation on $[n]^*$
where $D_{\mathfrak{p}}=p_t^{-1}$ and $\bar{s}\neq \hat{e}_n$.
Then, there exist $\bar{s}_{i-1}<_{\bar{s}} \bar{s}_j <_{\bar{s}} \bar{s}_{k-1} \leq_{\bar{s}} \bar{s}_l$ such that
$$
\pi(\bar{s}_{i-1})=\bar{s}_{k-1}, \quad \pi(\bar{s}_l)=\bar{s}_j.
$$
\end{lemma}
\begin{proof} Since $\bar{s}\neq \hat{e}_n$, there exists $x\in [n]$ such that $x+1<_{\bar{s}} x$.
Assume $x=\bar{s}_{k-1}$ is the largest such integer and let $\bar{s}_{i}=x+1$.
Then, $\pi(\bar{s}_{i-1})=x=\bar{s}_{k-1}$ since $D_{\mathfrak{p}}(\pi(\bar{s}_{i-1}))
=\pi(\bar{s}_{i-1})+1=x+1$.
Between $\bar{s}_{i-1}$ and $x$, find the largest integer which is larger than $x$.
Since $x+1$ lies between $\bar{s}_{i-1}$ and $x$, this maximum exists and we denote it by $y$. Then we
have by construction
$$
\bar{s}_{i-1}<_{\bar{s}} \bar{s}_j=y <_{\bar{s}} x= \bar{s}_{k-1} <_{\bar{s}} y+1=\bar{s}_{l+1}.
$$
Therefore, $\pi(\bar{s}_{l})=D_{\mathfrak{p}}^{-1}(y+1)=y=\bar{s}_j$, whence the lemma. \qed
\end{proof}

Now we can derive the exact block-interchange distance formula obtained by Christie~\cite{chri2}.
\begin{theorem}[Christie~\cite{chri2}]\label{5thm2}
\begin{equation}
bid(s)= \frac{n+1-C(p_t\bar{s})}{2}.
\end{equation}
\end{theorem}
\begin{proof}
Let $\mathfrak{p}=(\bar{s},\pi)$ be a plane permutation on $[n]^*$
where $D_{\mathfrak{p}}=p_t^{-1}$ and $\bar{s}\neq \hat{e}_n$.
According to Lemma~\ref{5lem1}, we either have ${\bar{s}}_{i-1}<_{\bar{s}} {\bar{s}}_j <_{\bar{s}} {\bar{s}}_{k-1} <_{\bar{s}} {\bar{s}}_l$
such that we either have $\pi$-cycle
$$
({\bar{s}}_{i-1},{\bar{s}}_{k-1},\ldots {\bar{s}}_l,{\bar{s}}_j,\ldots)\quad \mbox{or} \quad ({\bar{s}}_{i-1},{\bar{s}}_{k-1},\ldots)({\bar{s}}_l,{\bar{s}}_j,\ldots),
$$
or ${\bar{s}}_{i-1}<_{\bar{s}} {\bar{s}}_j <_{\bar{s}} {\bar{s}}_{k-1} =_{\bar{s}} {\bar{s}}_l$ such that
we have the $\pi$-cycle $({\bar{s}}_{i-1},{\bar{s}}_{k-1},{\bar{s}}_j,\ldots)$.
For the former case, the determined $\chi_h$ is either Case~$c$ or Case~$e$ of Lemma~\ref{2lem5}.
For the latter case, the determined $\chi_h$ is Case~$2$ transpose of Lemma~\ref{2lem4}.
Therefore, no matter which case, we can always find a block-interchange to
increase the number of cycles by $2$. Then,  arguing as in Theorem~\ref{5thm1} completes
the proof.\qed
\end{proof}

Note that each block-interchange can be achieved by at most two transpositions, i.e.,
\begin{align*}
&[i,j,j+1,\ldots k-1,k, l]\rightarrow [k,l,j+1,\ldots k-1, i,j] \Longleftrightarrow \\
&[i,j,j+1,\ldots k-1,k, l]\rightarrow [k,l,i,j,j+1,\ldots k-1]\rightarrow [k,l,j+1,\ldots k-1, i,j].
\end{align*}
Then, we immediately obtain an upper bound for the transposition distance.

\begin{corollary}\cite{pev1}
\begin{align}
td(s)\leq n+1-C(p_t \bar{s}).
\end{align}
\end{corollary}

Furthermore, Zagier and Stanley's result mentioned earlier implies that

\begin{corollary}\label{5cor7}
Let $bid_k(n)$ denote he number of sequences $s$ on $[n]$ such that $bid(s)=k$.
Then,
\begin{equation}
bid_k(n)=\frac{2C(n+2,n+1-2k)}{(n+1)(n+2)}.
\end{equation}
\end{corollary}
\begin{proof}
Let
$$
k=bid(s)= \frac{n+1-C(p_t\bar{s})}{2}.
$$
The number of $s$ such that $bid(s)=k$ is equal to the number of permutation $\bar{s}$
such that $C(p_t \bar{s})=n+1-2k$. Then, applying Zagier and Stanley's result
completes the proof. \qed
\end{proof}

We note that the corollary above was also used by Bona and Flynn~\cite{bona} to compute the average number of
block-interchanges needed to sort permutations.

In view of Theorem~\ref{5thm1} and Theorem~\ref{5thm2}, we are now in position to answer
one of the optimization problems mentioned earlier.

\begin{theorem}\label{5thm3}
Let $\alpha$ be a permutation on $[n]$ and $n\geq 1$. Then we have
\begin{equation}
\max_{\gamma}\{|C(\alpha \gamma)-C(\gamma)|\}=n-C(\alpha),
\end{equation}
where $\gamma$ ranges over all permutations on $[n]$.
\end{theorem}

\begin{proof}
{\it Claim.} for arbitrary $s$, we have
\begin{equation}
\max_\gamma\{|C(p_t\bar{s} \gamma)-C(\gamma)|\}=n+1-C(p_t\bar{s}),
\end{equation}
where $\gamma$ ranges over all permutations on $[n]^*$.\\
To prove the Claim, we argue as in Theorem~\ref{5thm1}, that
$$
bid(s)\geq \frac{\max_{\gamma}\{|C(p_t\bar{s} \gamma)-C(\gamma)|\}}{2}
$$
holds. On the other hand, Theorem~\ref{5thm2} guarantees
$$
bid(s)=\frac{n+1-C(p_t\bar{s})}{2}=\left.\frac{|C(p_t\bar{s} \gamma)-C(\gamma)|}{2}\right|_{\gamma=(p_t\bar{s})^{-1}},
$$
which means the maximum is achieved when $\gamma=(p_t\bar{s})^{-1}$, whence the Claim.

We now use the fact that any even permutation $\alpha'$ on $[n]^*$ has a factorization into
two $(n+1)$-cycles. Assume $\alpha'=\beta_1\beta_2$ where $\beta_1, \beta_2$ are two
$(n+1)$-cycles, and $p_t=\theta \beta_1 \theta^{-1}$.
Then, we have
\begin{eqnarray*}
\max_{\gamma}\{|C(\alpha' \gamma)-C(\gamma)|\}&=&\max_{\gamma}\{|C(\theta \beta_1\beta_2\gamma\theta^{-1})-C(\theta\gamma\theta^{-1})|\}\\
&=&\max_{\gamma}\{|C(p_t\theta\beta_2\theta^{-1}\theta\gamma\theta^{-1})-C(\theta\gamma\theta^{-1})|\}\\
&=& n+1-C(p_t\theta\beta_2\theta^{-1})\\
&=& n+1- C(\beta_1\beta_2)=n+1-C(\alpha').
\end{eqnarray*}
So the theorem holds for even permutations.
Next we assume $\alpha$ is an odd permutation. If $C(\alpha)<n$, then we can always find a transposition $\tau$ (i.e., a cycle of length $2$)
such that $\alpha= \alpha' \tau$, $\alpha'$ is an even permutation and $C(\alpha)=C(\alpha')-1$. Thus,
\begin{align*}
\max_{\gamma}\{|C(\alpha \gamma)-C(\gamma)|\}&= \max_{\gamma}\{|C(\alpha' \tau\gamma)-C(\gamma)|\}\\
&=\max_{\gamma}\{|C(\alpha' \tau\gamma)-C(\tau\gamma)+C(\tau\gamma)-C(\gamma)|\}\\
& \leq \max_{\gamma}\{|C(\alpha' \tau\gamma)-C(\tau\gamma)|+|C(\tau\gamma)-C(\gamma)|\}\\
&=[n-C(\alpha')]+1=n-C(\alpha).
\end{align*}
Note that $|C(\alpha I)-C(I)|=n-C(\alpha)$, where $I$ is the the identity permutation.
Hence, we conclude that $\max_{\gamma}\{|C(\alpha \gamma)-C(\gamma)|\}=n-C(\alpha)$.
When $C(\alpha)=n$, i.e., $\alpha=I$, it is obvious that $\max_{\gamma}\{|C(I \gamma)-C(\gamma)|\}=0=n-C(\alpha)$.
Hence, the theorem holds for odd permutations as well.
This completes the proof.\qed
\end{proof}

\section{Reversal distance for signed permutations}
In this section, we consider the reversal distance for signed permutations, a problem
extensively studied in the context of genome evolution~\cite{pev2,pev3,yan} and references therein.
Lower bounds for the reversal
distance based on the breakpoint graph model were obtained in~\cite{pev2,pev3,pev4}.

In our framework the reversal distance problem can be expressed as a
block-interchange distance problem. A lower bound can be easily obtained in
this point of view, and the lower bound will be shown to be the exact reversal distance
for most of signed permutations.

Let $[n]^{-}=\{-1,-2,\ldots, -n\}$.
\begin{definition}
A signed permutation on $[n]$ is a pair $(a,w)$ where $a$ is
a sequence on $[n]$ while $w$ is a word of length $n$ on the
alphabet set $\{+,-\}$.
\end{definition}
Usually, a signed permutation is represented by a single sequence
$a_w=a_{w,1}a_{w,2}\cdots a_{w,n}$ where $a_{w,k}=w_ka_k$, i.e.,
each $a_k$ carries a sign determined by $w_k$.

Given a signed permutation $a=a_1a_2\cdots a_{i-1}a_i a_{i+1}\cdots a_{j-1}
a_j a_{j+1}\cdots a_n$ on $[n]$, a reversal $\varrho_{i,j}$ acting on
$a$ will change $a$ into
$$
a'=\varrho_{i,j}\diamond a =a_1a_2\cdots a_{i-1}(-a_j)(-a_{j-1})\cdots
(-a_{i+1})(-a_i)a_{j+1}\cdots a_n.
$$
The reversal distance $d_r(a)$ of a signed permutation $a$ on $[n]$ is the minimum
number of reversals needed to sort $a$ into $e_n=12\cdots n$.

For the given signed permutation $a$, we associate the sequence $s=s(a)$ as follows
$$
s=s_0s_1s_2\cdots s_{2n}=0a_1a_2\cdots a_n (-a_n)(-a_{n-1})\cdots (-a_2)(-a_1),
$$
i.e., $s_0=0$ and $s_{k}=-s_{2n+1-k}$ for $1\leq k\leq 2n$. Furthermore, such sequences
will be referred to as skew-symmetric sequences since we have $s_{k}=-s_{2n+1-k}$. A sequence
$s$ is called exact if there exists $s_i<0$ for some $1\le i\le n$.
The reversal distance of $a$ is equal to the block-interchange distance of $s$
into
$$
e_n^{\natural}=012\cdots n (-n)(-n+1)\cdots (-2) (-1),
$$
where only certain block-interchanges are allowed, i.e., only the actions $\chi_h$,
$h=(i,j,2n+1-j,2n+1-i)$ are allowed where $1\leq i\leq j\leq n$. Hereafter, we will denote
these particular block-interchanges on $s$ as reversals, $\varrho_{i,j}$.

Let
\begin{eqnarray*}
\tilde{s}&=&(s)=(0,a_1,a_2,\ldots,a_{n-1},a_n,-a_n,-a_{n-1},\ldots,-a_2,-a_1),\\
p_r&=&(-1,-2,\ldots,-n+1,-n,n,n-1,\ldots,2,1,0).
\end{eqnarray*}
A plane permutation of the form $(\tilde{s},\pi)$ will be called skew-symmetric.

\begin{theorem}\label{7thm1}
\begin{equation}\label{5eq1}
d_r(a)\geq \frac{2n+1-C(p_r\tilde{s})}{2}.
\end{equation}
\end{theorem}
\begin{proof} Since reversals are restricted block-interchanges, the reversal distance will be
bounded by the block-interchange distance without restriction. Theorem~\ref{5thm2} then implies Eq.~\eqref{5eq1}. \qed
\end{proof}

Our approach gives rise to the question of how potent the restricted block-interchanges are.
Is it difficult to find a block-interchange increasing the number of cycles by $2$ that is a
reversal (i.e., $2$-reversal)?

We will call a plane permutation $(\tilde{s},\pi)$ exact, skew-symmetric if
$\tilde{s}$ is exact and skew-symmetric. The following lemma will show that there is almost
always a $2$-reversal.

\begin{lemma}\label{7lem1}
Let $\mathfrak{p}=(\tilde{s},\pi)$ be exact and skew-symmetric on $[n]^*\cup[n]^{-}$,
where $D_{\mathfrak{p}}=p_r^{-1}$. Then, there always exist $i-1$ and $2n-j$ such that
\begin{equation}
\pi({s}_{i-1})={s}_{2n-j},
\end{equation}
where $0\leq i-1\leq n-1$ and $n+1\leq 2n-j \leq 2n$. Furthermore,
we have the following cases
\begin{itemize}
\item[(a)] If $s_{i-1}<_s s_{j}<_s s_{2n-j} <_s s_{2n+1-i}$, then
\begin{equation}
\pi({s}_{i-1})={s}_{2n-j}, \quad \pi(s_{j})=s_{2n+1-i}.
\end{equation}
\item[(b)] If $s_{j}<_s s_{i-1}<_s s_{2n+1-i}<_s s_{2n-j} $, then
\begin{equation}
\pi({s}_{i-1})={s}_{2n-j},\quad \pi(s_{j})=s_{2n+1-i}.
\end{equation}
\end{itemize}
\end{lemma}
\begin{proof} We firstly prove the former part.
Assume $s_{i}$ is the smallest negative element among the subsequence $s_1s_2\cdots s_n$.
If $s_{i}=-n$, then we have $s_{2n+1-i}=-s_i=n$ by symmetry. Since $D_{\mathfrak{p}}=p_r^{-1}$,
for any $k$, $s_{k+1}=p_r^{-1}(s_{k})=s_k+1$ where $n+1$ is interpreted as $-n$. Thus,
$\pi(s_{i-1})=D_{\mathfrak{p}}^{-1}(s_i)=D_{\mathfrak{p}}^{-1}(-n)=n=s_{2n+1-i}$.
Let $2n-j=2n+1-i$, then $2n-j\geq n+1$ and we are done.
If $s_{i}>-n$, then we have $\pi(s_{i-1})=D_{\mathfrak{p}}^{-1}(s_i)=s_{i}-1\geq -n$.
Since $s_{i}$ is the smallest negative element among $s_t$ for $1\leq t\leq n$,
if $s_{2n-j}=s_{i}-1< s_i$, then $2n-j\geq n+1$, whence the former part.

Using $D_{\mathfrak{p}}=p_r^{-1}$ and the skew-symmetry $s_k=-s_{2n+1-k}$,
we have in case of (a) the following situation in $\mathfrak{p}$ (only relevant entries are illustrated)
\begin{eqnarray*}
\left(\begin{array}{cccccccccccc}
i-1 & i &\cdots & j & j+1 & \cdots & 2n-j & 2n+1-j & \cdots & 2n+1-i \\\hline
s_{i-1}&(s_{2n-j}+1)&\cdots & s_{j}&-s_{2n-j}&\cdots & s_{2n-j}&-s_{j}& \cdots & (-s_{2n-j}-1) \\
s_{2n-j}&\diamondsuit & \cdots & (-s_{2n-j}-1)&\diamondsuit &\cdots &\diamondsuit&\diamondsuit &\cdots & \diamondsuit
\end{array}\right).
\end{eqnarray*}
Therefore, we have
\begin{eqnarray*}
\pi(s_{i-1})=s_{2n-j},\quad \pi(s_{j})=-s_{2n-j}-1=s_{2n+1-i}.
\end{eqnarray*}
Analogously we have in case of $(b)$ the situation
\begin{eqnarray*}
\left(\begin{array}{cccccccccccc}
j & j+1 & \cdots & i-1 & i &\cdots &  2n+1-i & 2n+2-i & \cdots & 2n-j \\\hline
s_{j}& -s_{2n-j}&\cdots & s_ {i-1}& s_{2n-j}+1 &\cdots & -s_{2n-j}-1 &-s_{i-1}& \cdots & s_{2n-j} \\
-s_{2n-j}-1 &\diamondsuit & \cdots & s_{2n-j}&\diamondsuit &\cdots &\diamondsuit&\diamondsuit &\cdots & \diamondsuit
\end{array}\right).
\end{eqnarray*}
Therefore, we have
\begin{eqnarray*}
\pi(s_{{i-1}})=s_{2n-j},\quad \pi(s_{j})=-s_{2n-j}-1=s_{2n+1-i}.
\end{eqnarray*}
This completes the proof. \qed
\end{proof}

{\bf Remark.} The pair $s_{i-1}$ and $s_{2n-j}$ such that $\pi({s}_{i-1})={s}_{2n-j}$ is not unique.
For instance, assume the positive integer $k$, $1 \leq k \leq n-1$, is not in the subsequence $s_1s_2
\cdots s_n$ but $k+1$ is, then $\pi^{-1}(k)$ and $k=D_{\mathfrak{p}}^{-1}(k+1)$ form such a pair.


Inspection of Lemma~\ref{2lem5} and Lemma~\ref{7lem1} shows that
there is almost always a $2$-reversal for signed permutations.
The only critical cases, not covered in Lemma~\ref{7lem1}, are
\begin{itemize}
\item The signs of all elements in the given signed permutation are positive.
\item Exact signed permutation which for $1\leq i \leq n$ and $ n+1\leq 2n-j$,
$\pi(s_{i-1})=s_{2n-j}$ iff $2n-j=2n+1-i$.
\end{itemize}
We proceed to analyze the latter case. Since $\pi(s_{i-1})=s_{2n+1-i}=-s_i$, we have
\begin{eqnarray*}
\left(\begin{array}{cccccccc}
s_{i-1}&s_{i}&\cdots& s_n&-s_n&\cdots &s_{2n+1-i}&s_{2n+2-i}\\
\pi(s_{i-1})&\diamondsuit&\cdots &\diamondsuit&\diamondsuit&\cdots &\diamondsuit&\diamondsuit
\end{array}\right)
=\left(\begin{array}{cccccccc}
s_{i-1}&s_{i}&\cdots& s_n&-s_n&\cdots &-s_{i}&-s_{i-1}\\
-s_{i}&\diamondsuit&\cdots &\diamondsuit&\diamondsuit&\cdots &\diamondsuit&\diamondsuit
\end{array}\right).
\end{eqnarray*}

Due to $D_{\mathfrak{p}}$, $D_{\mathfrak{p}}(-s_{i})=s_{i}=-s_{i}+1$ (note that $n+1$ is interpreted as $-n$).
The only situation satisfying this condition is that $s_{i}=-n$, i.e.,
the sign of $n$ in the given signed permutation is negative.
Then, we have $\pi(s_{i-1})=s_{2n-j}=s_{2n+1-i}=n$.
We believe that in this case Lemma~\ref{2lem4} (instead of Lemma~\ref{2lem5}) provides a $2$-reversal.
Namely, $s_{i-1}$ (i.e., the preimage of $n$),~$s_n$ and $n=s_{2n+1-i}$ will form a Case~$2$ transpose
in Lemma~\ref{2lem4}, which will be true if $n$ and $s_n$ are in the same cycle of $\pi$, i.e.,
$\pi$ has a cycle $(s_{i-1}, s_{2n+1-i},\ldots s_n, \ldots)$.
In order to illustrate this we consider
\begin{example}
\begin{eqnarray*}
\left(\begin{array}{ccccccccc}
0&-3&1&2&-4&4&-2&-1&3\\
-4&0&1&4&3&-3&-2&2&-1
\end{array}\right)&=&(0,-4,3,-1,2,4,-3)(1)(-2)\\
\left(\begin{array}{ccccccccc}
0&2&-4&-1&3&-3&1&4&-2\\
1&4&-2&2&-4&0&3&-3&-1
\end{array}\right)&=&(0,1,3,-4,-2,-1,2,4,-3)
\end{eqnarray*}
We inspect, that in the first case $s_{i-1}=2$,~$s_n=-4$ and $n=4$ form a Case~$2$ transpose of Lemma~\ref{2lem4}.
In the second case $s_{i-1}=2$,~$s_n=3$ and $n=4$ form again a Case~$2$ transpose of Lemma~\ref{2lem4}.
\end{example}

Therefore, we conjecture
\begin{conjecture}\label{7conj1}
Let $\mathfrak{p}=(\tilde{s},\pi)$ be exact, skew-symmetric on $[n]^*\cup[n]^{-}$
where $D_{\mathfrak{p}}=p_r^{-1}$ and suppose $\pi(s_{i-1})=s_{2n+1-i}$, where $1\leq i\leq n$.
Then, $n$ and $s_n$ are in the same cycle of $\pi$.
\end{conjecture}

Lemma~\ref{7lem1}, Conjecture~\ref{7conj1} and an analysis of the preservation of exactness
under $2$-reversals suggest, that for a random signed permutation, it is likely to be
possible to transform $s$ into $e_n^{\natural}$ via a sequence of $2$-reversals.
In fact, many examples, including Braga~\cite[Table $3.2$]{braga}, indicate that the lower bound of
Theorem~\ref{7thm1} gives the exact reversal distances.

Note that the lower bound obtained in~\cite{pev2,pev4} via the break point graph also provides
the exact reversal distance for most of signed permutations although the exact reversal distance
was formulated later in~\cite{pev3}.
Now we give a brief comparison of our formula Eq.~\eqref{5eq1} and the lower bound via break point graph.
The break point graph for a given signed permutation $a=a_1a_2\cdots a_n$ on $[n]$ can be obtained as follows:
replacing $a_i$ with $(-a_i) a_i$, and adding $0$ at the beginning of the obtained sequence while
adding $-(n+1)$ at the end of the obtained sequence,
in this way we obtain a sequence $b=b_0b_1b_2\cdots b_{2n}b_{2n+1}$ on $[n]^* \cup [n+1]^{-}$.
Draw a black edge between $b_{2i}$ and $b_{2i+1}$, as well as a grey edge between $i$ and $-(i+1)$
for $0\leq i \leq n$. The obtained graph is the break point graph $BG(a)$ of $a$.
Note that each vertex in $BG(a)$ has degree two so that it can be decomposed into disjoint cycles.
Denote the number of cycles in $BG(a)$ as $C_{BG}(a)$. Then, the lower bound via the break point graph is
\begin{align}\label{bg-r}
d_r(a) \geq n+1-C_{BG}(a).
\end{align}
Looking into this formula and the break point graph, we can actually formulate Eq.~\eqref{bg-r} into
the form similar to our lower bound.
Let $\theta_1$,~$\theta_2$ be the two involutions (without fixed points) determined by
the black edges and grey edges in the break point graph, respectively, i.e.,
\begin{align*}
\theta_1 &=(b_0,b_1)(b_2,b_3)\cdots (b_{2n},b_{2n+1}),\\
\theta_2 &=(0,-1)(1,-2)\cdots (n,-n-1).
\end{align*}
It is not hard to observe that $C_{BG}(a)=\frac{C(\theta_1 \theta_2)}{2}$.
Therefore, we have
\begin{proposition}
\begin{align}\label{bp-r2}
d_r(a)\geq \frac{2n+2- C(\theta_1 \theta_2)}{2}.
\end{align}
\end{proposition}

Since both our lower bound Eq.~\eqref{5eq1} and the lower bound Eq.~\eqref{bp-r2} provide
the exact reversal distance for most of signed permutations, that suggests, for most of signed permutations,
$$
C(p_r \tilde{s})= C(\theta_1\theta_2)-1=2 C_{BG}(a)-1.
$$
Is it always true?

At last, we present the following generalization of Conjecture~\ref{7conj1}:

\begin{conjecture}\label{7conj2}
Let $\mathfrak{p}=(\tilde{s},\pi)$ be skew-symmetric on $[n]^*\cup[n]^{-}$
where $D_{\mathfrak{p}}=p_r^{-1}$.
Then, $n$ and $s_n$ are in the same cycle of $\pi$.
\end{conjecture}

%

\begin{acknowledgements}
We acknowledge the financial support of the Future and Emerging
Technologies (FET) programme within the Seventh Framework Programme (FP7) for
Research of the European Commission, under the FET-Proactive grant agreement
TOPDRIM, number FP7-ICT-318121.
\end{acknowledgements}

\end{document}